\documentclass[11pt]{amsart}
\usepackage[utf8]{inputenc}
\usepackage[total={6.5in,9in},top=1in, left=1in, right=1in, bottom=1in]{geometry}
\usepackage{color,amsmath,amssymb,mathtools,graphicx}
\usepackage{amsfonts}
\usepackage{libertine}
\usepackage[libertine]{newtxmath}
\usepackage{bbold}
\usepackage{epsfig,fancyhdr}
\usepackage{dsfont} 
\usepackage{graphicx,float}
\usepackage{graphicx}
\usepackage{epsfig,color,fancyhdr,setspace}
\usepackage{epstopdf}
\usepackage{wrapfig}
\usepackage{soul}
\usepackage{import}
\usepackage{lineno}
\usepackage{stackrel}
\usepackage[dvipsnames]{xcolor}

\usepackage[hidelinks,pagebackref,pdftex]{hyperref}

\usepackage[abs]{overpic}
\usepackage[center]{caption}
\usepackage{subcaption}
\usepackage{pst-node}
\usepackage{tikz-cd}
\usepackage{enumitem}
\counterwithin{figure}{section}
\usepackage{tikz-cd}
\usepackage{wasysym}
\usepackage{marvosym}
\usepackage{extarrows}
\usepackage{amsmath}
\usepackage{graphicx,amsmath,mathtools,float}
\usepackage{setspace}
\graphicspath{{./images/}}
\usepackage{stackrel}
\usepackage{yfonts}
\usepackage{enumitem, hyperref}\makeatletter
\def\namedlabel#1#2{\begingroup
    #2%
    \def\@currentlabel{#2}%
    \phantomsection\label{#1}\endgroup
}
\makeatother
\usepackage{type1cm} 
\usepackage{blindtext}
\usepackage{scalerel,stackengine}
\usepackage{tikz}
\usepackage{mathtools}
\usepackage{amssymb}
\usepackage{tikz}
\usetikzlibrary{shapes,snakes}

\newcommand\quotient[2]{
	\mathchoice
	{% \displaystyle
		\text{\raise1ex\hbox{$#1$}\Big/\lower1ex\hbox{$#2$}}%
	}
	{% \textstyle
		#1\,/\,#2
	}
	{% \scriptstyle
		#1\,/\,#2
	}
	{% \scriptscriptstyle  
		#1\,/\,#2
	}
}

\usepackage{tocbasic}
\newtheorem{theorem}{Theorem}[section]
\newtheorem{lemma}[theorem]{Lemma}
\newtheorem{definition}[theorem]{Definition}
\newtheorem{example}[theorem]{Example}
\newtheorem{proposition}[theorem]{Proposition}
\newtheorem{remark}[theorem]{Remark}

\newtheorem{corollary}[theorem]{Corollary}
\newtheorem{conjecture}[theorem]{Conjecture}

\newtheorem*{namedtheorem}{\theoremname}
\newcommand{\theoremname}{testing}
\newenvironment{named}[1]{\renewcommand{\theoremname}{#1}\begin{namedtheorem}}{\end{namedtheorem}}

\title{Chebyshev polynomials and Gram determinants from the M\"obius band}

\author{Anthony Christiana}
\address{Department of Mathematics, The George Washington University, Washington DC, USA}
\email{{\rm \textcolor{blue}{ajchristiana@gwmail.gwu.edu}}}

\author{Dionne Ibarra}
\address{School of Mathematics, Monash University, VIC 3800, Australia.}
\email{{\rm \textcolor{blue}{dionne.ibarra@monash.edu}}}

\author{Gabriel Montoya-Vega}
\address{Department of Mathematics, The Graduate Center CUNY, NY, USA, and \newline \indent Department of Mathematics, University of Puerto Rico-R\'io Piedras, San Juan, PR}
\email{{\rm \textcolor{blue}{gabrielmontoyavega@gmail.com}}}

\subjclass[2020]{Primary: 57K10. Secondary: 57K31.}
\keywords{Gram determinants, Mersenne numbers, knot theory, Chebyshev polynomials, knots and links, unorientable surfaces.}

\begin{document}
\begin{abstract}
This article explores the connection between Chebyshev polynomials and knot theory, specifically in relation to Gram determinants. We reveal intriguing formulae involving the Chebyshev polynomial of the first and second kind. In particular we show that for Mersenne numbers, $M_k=2^k-1$ where $k\geq 2$, the $M_k$-th Chebyshev polynomial of the second kind is the product of Chebyshev polynomials of the first kind. We then discuss the Gram determinant of type $(Mb)_1$, restate the conjecture of its closed formula in terms of mostly products of Chebyshev polynomials of the second kind, and prove a factor of the determinant that supports the conjecture. We also showcase an algorithm for calculating the Gram determinant's corresponding matrix. Furthermore, we restate Qi Chen's conjectured closed formula for the Gram determinant of type Mb and discuss future directions.
\end{abstract}

\maketitle

\tableofcontents

\section{Introduction}

Chebyshev polynomials, named after Pafnuty Lvovich Chebyshev, play an important role in modern knot theory. Shortly after the discovery of the Jones polynomial and its generalizations, J\'ozef H. Przytycki made an intriguing observation regarding the behavior of the polynomial when a link is altered by $t_k$-moves \cite{Prz}. The observation revealed that the Jones polynomial exhibited distinctive patterns under such modifications, offering new insights into its structural properties. Over time, this phenomenon was more precisely understood through the use of Chebyshev polynomials. They additionally appear in relation to the Jones-Wenzl idempotents and in Lickorish's construction of the Witten-Reshetikhin-Turaev $3$-manifold invariants \cite{Lic, RT, MH, KL}. Furthermore, Chebyshev polynomials play an important role in expressing the structure of skein modules of $3$-manifolds and skein algebras of trivial I-bundles over surfaces, see \cite{FG}. In this paper, we reveal interesting new relations between Mersenne numbers and the Chebyshev polynomials of the first and second kind. In particular,

\begin{named}{Corollary \ref{Coro:newrelation}}
     Let $M_k=2^k-1$ be the $k$-th Mersenne number.\footnote{We are not requiring $k$ to be prime.} For $k\geq 2$,
    $$ S_{M_k}(d) = \prod_{i=0}^{k-1} T_{2^i}(d).$$
\end{named} 

\ 

The development of the notion of Gram determinants in knot theory, is rooted in Edward Witten's speculation regarding the existence of a $3$-manifold invariant, associated with the Jones polynomial. Specifically, the construction of such an invariant by Nicolai Reshetikhin and Vladimir Turaev sparked a series of subsequent works by knot theorists exploring the connection between these determinants and the mathematical theory of knots. See \cite{IM2, PBIMW} for an introduction to Gram determinants in knot theory. Closed formulae have been proven for many Gram determinants in knot theory, in particular Type $A$, generalized Type $A$, and Type $B$. All of these closed formulae are written in terms of Chebyshev polynomials. In \cite{IM1} a conjectured closed formula for the Gram determinant of Type $(Mb)_1$ was presented and in \cite{IM2} a large portion of the conjectured factors were proven. The remaining factors were proven to be a determinant of a gram matrix denoted by $\tilde{G}_n^{Mb_{n,1}}$. Furthermore, the authors gave a conjectured closed formula for $\det (\tilde{G}_n^{Mb_{n,1}})$. In this paper we use the new relations between the Chebyshev polynomials to restate this conjecture and we prove a factor that supports it.

\begin{named}{Conjecture \ref{conjecturer}}
    For $n\geq 2$,
    $$\det (\tilde{G}_n^{Mb_{n,1}}) = \prod_{k=2}^n (d^2-4)^{\binom{2n}{n-k}}(S_{k-1}(d))^{2 \binom{2n}{n-k}}.$$
\end{named} 

\begin{named}{Theorem \ref{maintheorem}}
    For $n\geq 2$, $\det(\tilde{G}_{n}^{Mb_{n, 1}})$ is divisible by
    $S_1(d)^{2k}$,  where $k = \binom{2n}{n-2}$.
\end{named}
\ 

This article is organized as follows. Section \ref{DefisChebyPolys} presents the definition of the Chebyshev polynomials and some interesting properties. Then in Section \ref{GramMatrixDefi} the Gram determinant of type $(Mb)_1$ and a related Gram determinant together with its conjectured closed formula are introduced. A conjectured factor of this new Gram determinant is also proven in this section. In Section \ref{algorithm}, a novel program for the calculation of its matrix is showcased. Finally, Section \ref{FutDirec} comments on some possibilities for further research regarding the proof of a conjecture concerning the structure of the Gram determinant of type $Mb$.
\section{Chebyshev polynomials}\label{DefisChebyPolys}

In this section we present the definitions of Chebyshev polynomials following \cite{PBIMW}.

\begin{definition}\label{DefiChebyFirst}
  The Chebyshev polynomial of the first kind, $T_n(d)$, is defined by the initial conditions $T_0(d) = 2$, $T_1(d) = d$, and the recursive relation $T_n(d) = dT_{n-1}(d) - T_{n-2} (d)$. 
\end{definition}  

Observe that using the same recursive relation, written in the form $dT_n(d)=T_{n-1}(d)+T_{n+1}(d)$ we can extend the Chebyshev polynomial $T_n(d)$ to negative indexes and obtain $T_{-n}(d)= T_{n}(d)$.

\begin{definition}\label{DefiChebySecond}
  We define the Chebyshev polynomial of the second kind $S_n(d)$ by the initial conditions $S_0 (d) = 1$, $S_1 (d) = d$, and the recursive relation $S_n (d) = dS_{n-1}(d) - S_{n-2} (d)$.
  \end{definition}

$S_n(d)$ can also be extended to negative indexes and we get $S_{-n}(d)=-S_{n-1}(d)$ with $S_{-1}=0$.

The Chebyshev polynomials have a nice well-known product to sum formula:
\begin{eqnarray}
   T_m(d)T_n(d) &=& T_{m+n}(d) + T_{|m-n|}(d) \\
    S_n(d) S_m(d) &=& S_{|n-m|}(d)+ S_{|n-m|+2}(d)+ \dots + S_{n+m}(d). \label{equation:cheby2}
\end{eqnarray}

From Equation \ref{equation:cheby2}, we have 
\begin{equation} \label{equation:cheby2niceeq}
    S_{n}(d)S_m(d) = S_{m-1}(d)S_{n-1}(d)+S_{n+m}(d).
\end{equation}

\begin{lemma}\label{Lemma:chebyshev} For $n\geq 1$,
    $T_{4n}(d) -2 =  (T_n(d))^2((T_n(d))^2-4)$
    and 
    $ T_{2n}(d) -2 = ((T_n(d))^2-4).$
    In particular, the constant coefficient of $T_{4n}(d) -2$ is zero.  
\end{lemma}

\begin{proof}
    By the product to sum formula, $T_a(d)T_b(d)=T_{a+b}(d)+T_{|a-b|}(d)$, we immediately have
    $$T_{2n}(d) = (T_n(d))^2-2.$$
    Therefore, 
\begin{eqnarray*}
    T_{4n}(d)-2 &=& (T_{2n}(d))^2-4 \\
    &=& ((T_n(d))^2-2)^2-4  \\
    &=& (T_n(d))^2((T_n(d))^2-4).
    \end{eqnarray*}
\end{proof}

A direct application of Lemma \ref{Lemma:chebyshev} $(n-1)$-times yields the following. 
\begin{corollary}\label{coro:chebyT}
    For $n>1$
$$ (T_{2^n}(d))^2-4 = ((T_{2}(d))^2-4) \prod_{i=1}^{n-1} (T_{2^i}(d))^2$$
and 
$$T_{2^n}-2 = ((T_{1}(d))^2-4) \prod_{i=0}^{n-2} (T_{2^i}(d))^2.$$
\end{corollary}

\begin{lemma}\label{lemma:restatement2}
    For $n>1$,
$$ (T_n(d))^2-4 = (d^2-4)(S_{n-1}(d))^2.$$
\end{lemma}

\begin{proof} 
From Chapter 6 of \cite{PBIMW} we know the following formula that relates the Chebyshev polynomials of the first and second kind:
$$(T_n(d))^2-d^2=S_n(d)S_{n-2}(d)(d^2-4).$$ Then, 

\begin{eqnarray*}
    (T_n(d))^2-4  &=& S_n(d)S_{n-2}(d)(d^2-4)+d^2-4 \\
    &=& (d^2-4)(S_n(d)S_{n-2}(d)+1)\\
    &=& (d^2-4)(S_n(d)S_{n-2}(d)+S_0(d)) \hspace{3cm} (\ast).
\end{eqnarray*}

Observe that by the product-to-sum formula, we have:
\begin{eqnarray*}
    S_n(d)S_{n-2}(d) &=&S_{|n-2-n|}(d)+S_{|n-2-n|+2}(d)+\cdots+S_{n+n-2}(d)\\
    &=& S_{2}(d)+S_{4}(d)+\cdots+S_{2n-2}(d).
\end{eqnarray*}

Thus, 
\begin{eqnarray*}
    S_n(d)S_{n-2}(d)+S_0(d) &=& S_0(d)+S_2(d)+\cdots+S_{2n-2}(d)\\
    &=&  S_n(d)S_n(d)-S_{2n}(d).
\end{eqnarray*}

Then Equation $(\ast)$ above becomes:
$$(T_n(d))^2-4=(d^2-4)(S_n(d)S_n(d)-S_{2n}(d))=(d^2-4)(S^2_n(d)-S_{2n}(d)) \hspace{3cm} (\ast\ast).$$

Observe that using the product-to-sum formulas:
$$S_m(d)S_n(d)=S_{|n-m|}(d)+S_{|n-m|+2}(d)+\cdots+S_{n+m}(d)$$
and
$$S_{m-1}(d)S_{n-1}(d)=S_{|n-m|}(d)+S_{|n-m|+2}(d)+\cdots+S_{n+m-2}(d),$$
we obtain the following equation
$$S_m(d)S_n(d)=S_{m-1}(d)S_{n-1}(d)+S_{n+m}(d),$$
which we can use to rewrite $(S_n(d))^2$ as:
$$(S_n(d))^2=S_{n-1}(d)S_{n-1}(d)+S_{2n}(d),$$
which implies that $(S_n(d))^2-S_{2n}(d)=(S_{n-1}(d))^2$.
Thus, Equation $(\ast\ast)$ becomes
$$(T_n(d))^2-4=(d^2-4)(S_{n-1}(d))^2.$$
\end{proof}

As a consequence to Lemmas \ref{Lemma:chebyshev} 
 and \ref{lemma:restatement2}, we have that the $M_k$-th Chebyshev polynomial of the second kind is the product of Chebyshev polynomials of the first kind, where $M_k = 2^k-1$ is the $k$-th Mersenne number. 

\begin{corollary}\label{Coro:newrelation}  Let $M_k=2^k-1$ be the $k$-th Mersenne number. For $k\geq 2$,
    $$ S_{M_k}(d) = \prod_{i=0}^{k-1} T_{2^i}(d).$$
\end{corollary}
\begin{proof}
    By Lemma \ref{lemma:restatement2}, 
    $$ (T_{2^k}(d))^2-4 = (d^2-4)(S_{2^k-1}(d))^2.$$
    By Lemma \ref{Lemma:chebyshev} and Corollary \ref{coro:chebyT}, 
    $$(T_{2^k}(d))^2-4 = T_{2^{k+1}}-2 = (d^2-4) \prod_{i=0}^{k-1} (T_{2^i}(d))^2.$$
\end{proof}
\section{The Gram determinant of type $(Mb)_1$}\label{GramMatrixDefi}

Przytycki introduced a novel concept known as the Gram determinant of type $Mb$ in 2008. This idea was developed as part of investigating all crossingless connections on a M\"obius band between $2n$ points on the boundary. To define this Gram determinant, he worked with a bilinear form constructed by identifying two Möbius bands along their boundaries \cite{BIMP}. The Gram determinant of type $(\mathit{Mb})_1$  was introduced in \cite{IM1} by using a subset of all crossingless connections on a M\"obius band between $2n$ points on the boundary that are related to crossingless connections in an annulus with $2n$ points on the outer boundary and either zero or two points on the inner boundary.

\begin{definition}\label{TypeMB1} 
Let $(\mathit{Mb}_{n})_1 = \mathit{Mb}_{n,0} \cup \mathit{Mb}_{n,1}$ where $\mathit{Mb}_{n,0} = \{ m_1, \dots, m_{\binom{2n}{n}}\}$ is the set of all diagrams of crossingless connections between $2n$ marked points on the boundary of $\mathit{Mb} \ \hat{\times} \ \{0\}$ whose arcs do not intersect the crosscap and $\mathit{Mb}_{n,1} = \{ m_1, \dots, m_{\binom{2n}{n-1}}\}$ is the set of all diagrams of crossingless connections between $2n$ marked points on the boundary of  $\mathit{Mb} \ \hat{\times} \ \{0\}$ with exactly one curve intersecting the crosscap. 
Define a bilinear form $\langle \ , \ \rangle_{\mathit{Mb}}$ on the elements of $(\mathit{Mb}_{n})_1$ by using the same bilinear form as type $\mathit{Mb}$, as follows: $$ \langle \ , \ \rangle_{\mathit{Mb}} : \mathcal{S}_{2,\infty}(\mathit{Mb}\  \hat \times \ I, \{x_i\}_1^{2n}) \times \mathcal{S}_{2,\infty}(\mathit{Mb}\ \hat \times \  I, \{x_i\}_1^{2n}) \longrightarrow \mathbb{Z}[d,w,x,y,z].$$

Given $m_i, m_j \in (\mathit{Mb}_{n})_1$, identify the boundary component of $m_i$ with that of the inversion of $m_j$, respecting the labels of the marked points. The result is an element in $Kb \ \hat{\times} \ I$ containing only disjoint simple closed curves. Then $\langle m_i , m_j\rangle_{\mathit{Mb}} :=  d^mx^ny^kz^lw^h$ where $m,n,k,l$ and $h$ denote the number of these curves, respectively.

The Gram matrix of type $(\mathit{Mb})_1$ is defined as $\ G_n^{(\mathit{Mb})_1} = (\langle m_i , m_j\rangle_{\mathit{Mb}})_{1 \leq i, j \leq \binom{2n}{n-1} +\binom{2n}{n}}$ and its determinant $D_n^{(\mathit{Mb})_1}$ is called the Gram determinant of type $(\mathit{Mb})_1$.
\end{definition}

In \cite{IM2}, the authors proved a large portion of the conjectured factors of $\det(G_{n}^{(Mb)_1})$, then they related the remaining factors as the determinant of a new Gram matrix, denoted by $\tilde{G}_{n}^{Mb_{n, 1}}$, defined below. 

\begin{definition}
    Consider the set $Mb_{n,1}$ containing the collection of all diagrams of crossingless connections between $2n$ marked points on the boundary of $Mb \hat{\times} \{0\}$ with exactly one curve intersecting the crosscap. Let $G_{n}^{Mb_{n,1}}$ be the Gram matrix defined on the set $Mb_{n,1}$ using the bilinear form $\langle \ , \ \rangle_{Mb}$,

    $$ G_{n}^{Mb_{n, 1}} = (\langle m_i, m_j \rangle_{Mb})_{1 \leq i,j \leq \binom{2n}{n-1}},$$
    where $m_i, m_j \in Mb_{n,1}$. Denote by $\tilde{G}_{n}^{Mb_{n, 1}}$ the Gram matrix obtained from substituting $y=0$ and $w=1$ into $G_{n}^{Mb_{n, 1}}$,
    $$\tilde{G}_{n}^{Mb_{n, 1}} = G_{n}^{Mb_{n, 1}}(y=0, w=1, d, z, x).$$
\end{definition}

In \cite{IM2} the following conjecture  was presented. Furthermore, the authors proved that $\det(\tilde{G}_{n}^{Mb_{n, 1}}) $ divides $\det(G_n^{(Mb)_1})$.

\begin{conjecture}\label{newconjecture}\cite{IM2} For $n \geq 2$,
    $$\det(\tilde{G}_{n}^{Mb_{n, 1}}) = \prod\limits_{k=2}^n (T_{2k}(d)-2)^{\binom{2n}{n-k}}.$$
\end{conjecture}

Lemma \ref{lemma:restatement2} allows us to restate the conjectured closed formula for $\det(G_n^{(Mb)_1})$ and Conjecture \ref{newconjecture} as follows.

\begin{conjecture}\label{Conjecture}\cite{IM1}\

Let $R = \mathbb{Z}[A^{\pm 1},w,x,y,z].$ Then,
the Gram determinant of type $(Mb)_1$ for $n \geq 1$, is:
\begin{eqnarray*}
D^{(Mb)_1}_n &=&  \left[(d-z)((d + z)w -2xy) \right]^{\binom{2n}{n-1}} \prod_{k=2}^n (T_k(d)^2-z^2)^{\binom{2n}{n-k} }\prod_{k=2}^n (d^2-4)^{\binom{2n}{n-k}}(S_{k-1}(d))^{2 \binom{2n}{n-k}}.
\end{eqnarray*}

\end{conjecture}

\begin{conjecture}\label{conjecturer}\cite{IM2} For $n\geq 2$,
    $$\det (\tilde{G}_n^{Mb_{n,1}}) = \prod_{k=2}^n (d^2-4)^{\binom{2n}{n-k}}(S_{k-1}(d))^{2 \binom{2n}{n-k}}.$$
\end{conjecture}

This conjecture was confirmed for the cases $n=2,3$, and $4$ in \cite{IM2}. We will prove a factor of $\det (\tilde{G}_n^{Mb_{n,1}})$ that supports the conjecture.

\begin{theorem}\label{maintheorem} For $n\geq 2$, $\det(\tilde{G}_{n}^{Mb_{n, 1}})$ is divisible by
    $S_1(d)^{2k}$,  where $k = \binom{2n}{n-2}$.
\end{theorem}
\begin{proof}
Consider the class of elements $\mathcal{M}$ consisting of four crossingless connections that differ only between the two arcs, shown in Figure \ref{fig:class}, where in each case the arcs are attached to four fixed points. Let $X_n = \{ x_1, \dots, x_{2n}\}$ be the collection of the $2n$ marked points on the boundary and let $X = \{x_i, x_j, x_k, x_s\}$ be the four fixed points in the class $\mathcal{M}$. \\

There are $\binom{2n}{n-2}$ classes. Let $\pmb m$ be a crossingless connection constructed from choosing $n-2$ marked points on the boundary and attaching arcs to them as discussed in the ``children playing a game" proof of Theorem 7.2.3 in \cite{PBIMW}. This construction gives $\binom{2n}{n-2}$ distinct connections between $2(n-2)$ points, denote this collection by $\mathcal{M}'$. Denote the collection of $2(n-2)$ points connected by arcs in $\pmb m$ by $X_n \backslash X$, then points in $X$ have a crossingless path to the crosscap. Each class $\mathcal{M}$ is constructed from a distinct element $\pmb m \in \mathcal{M}'$: The four elements are constructed by connecting the points in $X$ as shown in Figure \ref{fig:class}.
    \begin{figure}[ht]
    \centering
    $$ \left\{\vcenter{\hbox{
\begin{overpic}[scale = 1.5]{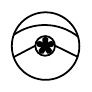}
\put(-1, 20){$x_j$}
\put(-1, 40){$x_i$}
\put(58, 20){$x_k$}
\put(58, 40){$x_s$}
\put(29, 0){$m_1$}
\end{overpic} }} \quad \vcenter{\hbox{
\begin{overpic}[scale = 1.5]{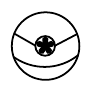}
\put(-1, 20){$x_j$}
\put(-1, 40){$x_i$}
\put(58, 20){$x_k$}
\put(58, 40){$x_s$}
\put(29, 0){$m_2$}
\end{overpic} }} \quad \vcenter{\hbox{
\begin{overpic}[scale = 1.5]{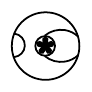}
\put(-1, 20){$x_j$}
\put(-1, 40){$x_i$}
\put(58, 20){$x_k$}
\put(58, 40){$x_s$}
\put(29, 0){$m_3$}
\end{overpic} }} \quad \vcenter{\hbox{
\begin{overpic}[scale = 1.5]{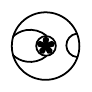}
\put(-1, 20){$x_j$}
\put(-1, 40){$x_i$}
\put(58, 20){$x_k$}
\put(58, 40){$x_s$}
\put(29, 0){$m_4$}
\end{overpic} }}\right\}$$
    \caption{The class of elements $\mathcal{M}$ that differ by two curves.}
    \label{fig:class}
\end{figure}
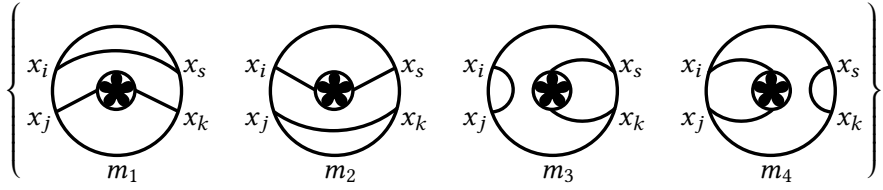

Notice that $m_1$ and $m_2$ are constructed from each other by swapping which arc intersects the crosscap. Similarly for $m_3$ and $m_4$. \\

We create a $4 \times 4$ matrix from the bilinear form of the four elements. Notice that, the $2n-2$ arcs that are not connected to $x_i, x_j, x_k,$ and $x_s$ have the same connection in all elements in $\mathcal{M}$. Furthermore, the connections made between the points $X_n \backslash X$ must all form homotopically trivial curves. So for $n \geq 3$, the entries of this matrix will have a multiple of a monomial $u=d^p$ for some $p\geq 0$. For $n=2$, $u=1$. 

\begin{equation}
    G = (\langle m_i, m_j \rangle )_{1\leq i,j \leq 4}=\begin{pmatrix}
        d u & 0 & u& u\\
        0 & d u & u& u \\
        u & u & d u & 0\\
        u & u & 0 & d u \\
    \end{pmatrix}.
\end{equation}

$G^*$ is constructed from $G$ by replacing $c_1$ with $c_1-c_2$ as well as $c_3$ with  $c_3-c_4$. Notice that in $G^*$ columns $c_1$ and $c_3$ each have a $d$ factor. 

\begin{equation}
    G^* = \begin{pmatrix}
        d u & 0 & 0& u\\
        -d u & d u & 0& u \\
        0 & u & d u & 0\\
        0 & u & -d u  & d u \\
    \end{pmatrix}.
\end{equation}

Now, take an arbitrary element $m \in Mb_{n,1}$, then  $\langle m, m_1\rangle$ is equal to $0, u_m$, or $d u_m'$ for some monomials $u_m$ and $u_m'$ where the factors $0, 1$, or $d$ are determined by the arcs in $m$ attached to $x_i, x_j, x_k$ and $x_s$. \\

First note that if $m$ is in a different class of elements $\mathcal{M}'$ where elements in the class differ by two arcs attached to points in $X$, then the same argument can be made with a different monomial. So we will assume that in $m$ there exists at least one arc with one boundary point attached to a point in $X$ and another boundary point attached to $X_n - X$. Also note that a change from $m_1$ to $m_2, m_3,$ or $m_4$ will only affect the simple closed curves formed from arcs attached to the points in $X$. We have two cases to consider after taking the bilinear form of $m$ and $m_1$: \\

\begin{enumerate}
\item[Case 1] Suppose one simple closed curve $\alpha$ is formed from the arcs attached to the points in $X$. Then either no arcs in $m$ forming $\alpha$ intersect the crosscap, or exactly one does.
\begin{enumerate}
    \item Suppose no arcs in $m$ forming $\alpha$ intersect the crosscap, then $\alpha$ intersects the crosscap in $m_1$ and there exists another simple closed curve $\beta$ that intersects the crosscap in $m$. Therefore, $\langle m, m_1\rangle=0$. The bilinear form is left unchanged by the skein move $m_1$ to $m_2$ because the skein move doesn't change $\alpha$ or $\beta$. Therefore,  $\langle m, m_2\rangle=0$. A change from $m_1$ to $m_3$ or $m_4$ keeps $\beta$ unchanged so there must exist an arc that intersects the crosscap in $m_1$ but not in $m$. So, $\langle m, m_3\rangle=\langle m, m_4\rangle=0$.

     \begin{figure}[ht]
    \centering
   \begin{subfigure}{.23\textwidth}
\centering
$ \vcenter{\hbox{
\begin{overpic}[scale = 1]{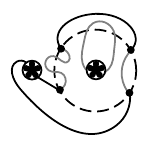}
\end{overpic} }} $
\caption{$\langle m, m_1\rangle = xyu_m=0$.} \label{case1a1}
\end{subfigure}
 \begin{subfigure}{.23\textwidth}
\centering
$ \vcenter{\hbox{
\begin{overpic}[scale = 1]{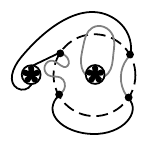}
\end{overpic} }} $
\caption{$\langle m, m_2\rangle = xyu_m=0$.} \label{case1a11}
\end{subfigure}
 \begin{subfigure}{.23\textwidth}
\centering
$ \vcenter{\hbox{
\begin{overpic}[scale = 1]{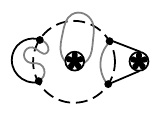}
\end{overpic} }} $
\caption{$\langle m, m_3\rangle = dxyu_m=0$.} \label{case1a12}
\end{subfigure}
 \begin{subfigure}{.23\textwidth}
\centering
$ \vcenter{\hbox{
\begin{overpic}[scale = 1]{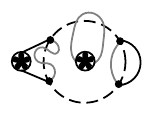}
\end{overpic} }} $
\caption{$\langle m, m_4\rangle = dxyu_m=0$.} \label{case1a13}
\end{subfigure}
    \caption{Example of Case 1(a).}
    \label{fig:case1a}
\end{figure}
    
    \item Suppose exactly one arc in $m$ forming $\alpha$ intersects the crosscap, then $\alpha$ intersects the crosscap twice. Therefore, $\langle m, m_1\rangle=u_m$, where $u_m$ is a monomial and all of the curves formed from the bilinear form is left unchanged by the skein move from $m_1$ to $m_2$. That is,  $\langle m, m_2\rangle=u_m$. A change from $m_1$ to $m_3$ or $m_4$ will change $\alpha$ to two simple closed curves, one from arcs attached to $\{x_i, x_j\}$ and the second from arcs attached to $\{x_k, x_s\}$, denote them by $\gamma_{i,j}$, $\gamma_{k,s}$, respectively.
    \begin{enumerate}
        \item Suppose $\gamma_{i,j}$ intersects the crosscap in $m$. Then $\gamma_{k,s}$ will intersect the crosscap in $m_3$ yielding $\langle m, m_3 \rangle =0$. Furthermore, $\gamma_{i,j}$ will intersect the crosscap in $m_4$ and $\gamma_{k,s}$ will be a homotopically trivial curve. Therefore,  $\langle m, m_4 \rangle = d u_m$.
         \item Suppose $\gamma_{k,s}$ intersects the crosscap in $m$. Then $\gamma_{k,s}$ will intersect the crosscap in $m_3$ and $\gamma_{k,s}$ will be a homotopically trivial curve. Therefore, $\langle m, m_3 \rangle = d u_m$. Furthermore, $\gamma_{i,j}$ will intersect the crosscap in $m_4$ which implies that  $\langle m, m_4 \rangle = 0$.
    \end{enumerate}
     \begin{figure}[ht]
    \centering
   \begin{subfigure}{.23\textwidth}
\centering
$ \vcenter{\hbox{
\begin{overpic}[scale = 1]{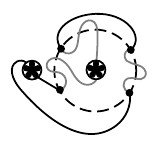}
\end{overpic} }} $
\caption{$\langle m, m_1\rangle = wu_m=u_m$.} \label{case1ai1}
\end{subfigure}
 \begin{subfigure}{.23\textwidth}
\centering
$ \vcenter{\hbox{
\begin{overpic}[scale = 1]{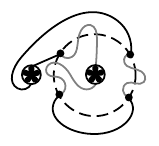}
\end{overpic} }} $
\caption{$\langle m, m_2\rangle = wu_m=u_m$.} \label{case1ai2}
\end{subfigure}
 \begin{subfigure}{.23\textwidth}
\centering
$ \vcenter{\hbox{
\begin{overpic}[scale = 1]{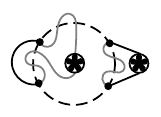}
\end{overpic} }} $
\caption{$\langle m, m_3\rangle = xyu_m=0$.} \label{case1ai3}
\end{subfigure}
 \begin{subfigure}{.24\textwidth}
\centering
$ \vcenter{\hbox{
\begin{overpic}[scale = 1]{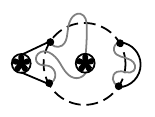}
\end{overpic} }} $
\caption{$\langle m, m_4\rangle = wdu_m=du_m$.} \label{case1ai4}
\end{subfigure}
    \caption{Example of Case 1(b)(i).}
    \label{fig:case1b}
\end{figure}
  
\end{enumerate}

\item[Case 2] Suppose two simple closed curves $\alpha$ and $\beta$ are formed from the points in $X$. Then one simple closed curve, say $\alpha$ is formed from arcs in $m$ attached to $\{x_i, x_s\}$ and $\beta$ is formed from arcs in $m$ attached to $\{x_j, x_k\}$.
\begin{enumerate}
    \item Suppose $\alpha$ intersects the crosscap in $m$, then $\beta$ intersects the crosscap in $m_1$. Therefore, $\langle m, m_1\rangle=0$. The skein move from $m_1$ to $m_2$ changes the $y$-curve to a homotopically trivial curve and the $x$-curve to a $w$-curve. This implies $\langle m, m_2\rangle= d u_m'$, where $u_m'$ is a monomial. Furthermore, a change from $m_1$ to $m_3$ or $m_4$ changes $\alpha$ and $\beta$ to one simple closed curve that intersects both crosscaps while leaving everything else unchanged. Therefore,  $\langle m, m_3\rangle=\langle m, m_4\rangle=u_m'$. 
    
  \begin{figure}[ht]
    \centering
   \begin{subfigure}{.23\textwidth}
\centering
$ \vcenter{\hbox{
\begin{overpic}[scale = 1]{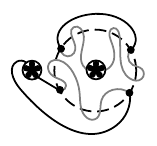}
\end{overpic} }} $
\caption{$\langle m, m_1\rangle = xyu_m'=0$.} \label{case2a1}
\end{subfigure}
 \begin{subfigure}{.23\textwidth}
\centering
$ \vcenter{\hbox{
\begin{overpic}[scale = 1]{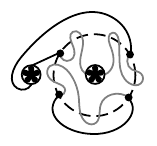}
\end{overpic} }} $
\caption{$\langle m, m_2\rangle = dwu_m'=du_m'$.} \label{case2a2}
\end{subfigure}
 \begin{subfigure}{.23\textwidth}
\centering
$ \vcenter{\hbox{
\begin{overpic}[scale = 1]{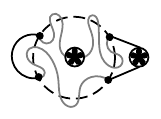}
\end{overpic} }} $
\caption{$\langle m, m_3\rangle = wu_m'=u_m'$.} \label{case2a3}
\end{subfigure}
 \begin{subfigure}{.23\textwidth}
\centering
$ \vcenter{\hbox{
\begin{overpic}[scale = 1]{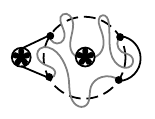}
\end{overpic} }} $
\caption{$\langle m, m_4\rangle = wu_m=u_m'$.} \label{case2a4}
\end{subfigure}
    \caption{Example of Case 2(a).}
    \label{fig:case2a}
\end{figure}

     \item Suppose $\beta$ intersects the crosscap in $m$, then it intersects the crosscap twice. This implies that the second simple closed curve is homotopically trivial. Therefore, $\langle m, m_1\rangle=du_m''$, where $u_m''$ is a monomial. The skein move from $m_1$ to $m_2$ changes the homotopically trivial curve to a $y$-curve and the $w$-curve to an $x$-curve. This implies $\langle m, m_2\rangle=0$. Furthermore, a change from $m_1$ to $m_3$ or $m_4$ changes $\alpha$ and $\beta$ to one simple closed curve that intersects both crosscaps while leaving everything else unchanged. Therefore,  $\langle m, m_3\rangle=\langle m, m_4\rangle=u_m''$.
     
     \begin{figure}[ht]
    \centering
   \begin{subfigure}{.24\textwidth}
\centering
$ \vcenter{\hbox{
\begin{overpic}[scale = 1]{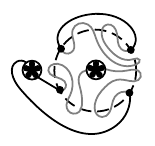}
\end{overpic} }} $
\caption{$\langle m, m_1\rangle = dwu_m''=du_m''$.} \label{case2b1}
\end{subfigure}
 \begin{subfigure}{.23\textwidth}
\centering
$ \vcenter{\hbox{
\begin{overpic}[scale = 1]{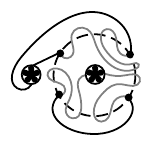}
\end{overpic} }} $
\caption{$\langle m, m_2\rangle = xyu_m''=0$.} \label{case2b2}
\end{subfigure}
 \begin{subfigure}{.23\textwidth}
\centering
$ \vcenter{\hbox{
\begin{overpic}[scale = 1]{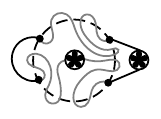}
\end{overpic} }} $
\caption{$\langle m, m_3\rangle = wu_m''=u_m''$.} \label{case2b3}
\end{subfigure}
 \begin{subfigure}{.23\textwidth}
\centering
$ \vcenter{\hbox{
\begin{overpic}[scale = 1]{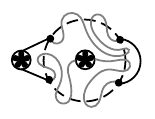}
\end{overpic} }} $
\caption{$\langle m, m_4\rangle = wu_m''=u_m''$.} \label{case2b4}
\end{subfigure}
    \caption{Example of Case 2(b).}
    \label{fig:case2b}
\end{figure}
\end{enumerate}
\end{enumerate}

Therefore, for any element $m \in Mb_{0,1}$ there exists monomials $u_m, u_m', u_m''$ such that the block 

$\langle m, m_1\rangle, \langle m, m_2\rangle, \langle m, m_3\rangle, \langle m, m_4\rangle$ is equal to one of the rows of the matrix in Equation \ref{eqn:blockmatrix}. 

\begin{equation}\label{eqn:blockmatrix}
    \begin{bmatrix}
        0 & 0 & 0 & 0\\
        u_m & u_m & d u_m & 0\\
        u_m & u_m & 0 & d u_m\\
        0 & d u_m' & u_m' & u_m' \\
        d u_m'' & 0 & u_m'' & u_m''
    \end{bmatrix}.
\end{equation}

  Thus the column operations $c_1$ to $c_1'=c_1-c_2$ and $c_3$ to $c_3'=c_3-c_4$ yield a zero or a monomial multiple of $d$. 
  
  \begin{equation}\label{eqn:blockmatrixcolumn}
    \begin{bmatrix}
        0 & 0 & 0 & 0\\
        0 & u_m & d u_m & 0\\
        0 & u_m & -d u_m & d u_m\\
        -d u_m' & d u_m' & 0 & u_m' \\
        d u_m'' & 0 & 0 & u_m''
    \end{bmatrix}
\end{equation}
  
  Therefore the new columns $c_1'$ and $c_3'$ are divisible by $d$. Furthermore, this is true for any class of elements in $\mathcal{M}$. Since there are $\binom{2n}{n-2}$ classes then $\det(\tilde{G}_n^{Mb_{n,1}})$ is divisible by $d^{2 \binom{2n}{n-2}}$. 
  
\end{proof}
%%%%%%%%%%%%%%%%%%%%%%%%%%%%%%%%%%%%%%%%%%%

\section{Program for Computing Bilinear Form}\label{algorithm}

The following summarizes a method for computing the bilinear form of elements in $(Mb)_1$ symbolically. These ideas are included in a Mathematica program which can be found in \cite{C}.

\subsection{Involutive Notation}

In \cite{IM1, IM2}, elements of $(Mb)_n$ are notated following the ``children playing a game" notation. These finite sequences of length $<n$ give a complete description of the crossingless connections between $2n$ points on the boundary of the M\"obius band. We contrast this notational style with \textit{involutive notation}, defined here.

\begin{remark}
Any element in $(Mb)_n$ can be written in \textit{involutive notation }, 

\[ m = (x_1  \ x_1') (x_2 \ x_2') \dots (x_k \ x_k')(x_{2k+1})(x_{2k+2}) \dots (x_{2n}) \]

where $x_i < x_j \iff i < j.$

The set of tuples $(x_i \ x_i')$ denote arcs which travel counterclockwise around the internal crosscap from vertex $x_i$ to vertex $x_i'$. Each singleton $(x_j)$ denotes a vertex connected to the crosscap by an arc, the set of which we call the \textit{fixed points}: 

\[ F(m) = \{x_i  \mid 2k+1 \leq i \leq 2n \}.  \]
\end{remark}

To write a program which computes $\langle m_1, m_2 \rangle$ for $m_1, m_2$ in involutive notation, we pass through a closely-related object: a graph called $G_{\langle m_1, m_2 \rangle}$.

\begin{definition}
For $m_1,m_2 \in (Mb)_n$,  define $G_{\langle m_1, m_2 \rangle}$ to be the graph with 

    \begin{itemize}
        \item $V(G_{\langle m_1, m_2 \rangle}) = \{i\}_{1=1}^{2n}$
        \item $E(G)=T \cup EF(m_1) \cup EF(m_2)$  where 
        
        \[ T=\{uv \mid (u \ v) \in m_1 \text{ or } (u \ v) \in m_2 \}, \]

        \[ EF(m_i) = \{f_i f_j \mid f_i,f_j \in F(m_1), \text{ and } j=\left(i+\frac{|F(m_i)|}{2}\right) \mod{|F(m_i)|} \}.\footnote{A word on cyclic group representatives: in this section, we take $[n]_{2n}$ as the canonical representative of the identity element of $\mathbb{Z}_{2n}.$ This choice can be avoided by relabelling the vertices along the boundary of the M\"obius band to have smallest vertex label $0$. We chose not to do this because 1) the CPaG notation has labelled vertices starting with 1, not 0, and 2) having smallest entry 1 is consistent with the standard notation for $S_{2n}$.}\]
    \end{itemize}
\end{definition}

The sets which comprise the edges of $G_{\langle m_1, m_2 \rangle}$ are 
\begin{itemize}
    \item $T$; edges between the labelled vertices which appear as entries of a transposition in $m_1$ or $m_2$. This edge set is in bijection with all arcs in $\langle m_1, m_2 \rangle$ disjoint from the crosscap.
    \item $EF(m_i)$; edges between vertices which are connected by an arc through the crosscap in $m_i$. The topology of the 
    M\"obius band ensures that the subset of vertices with connections through the crosscap are connected antipodally. 
    \item All three sets define a clear bijection
    \[ \{\text{edges of } G_{\langle m_1, m_2 \rangle} \} \longleftrightarrow \{ \text{arcs between vertices in } \langle m_1, m_2 \rangle \}.\]

    and thus 

    \[ \{ \text{simple closed curves in } KB\} \longleftrightarrow \{ \text{ disjoint cycles in }  G_{\langle m_1, m_2 \rangle}\}.\]
\end{itemize}

\begin{example}
    Let $m_1 = (2 \ 5)(3 \ 4)(1)(6)$,  $m_2 = (6 \ 1) (2)(3)(4)(5),$. Then 
    \[ T = \{ \{2,5\}, \{3,4\}, \{6,1\}\} \]
    \[ EF(m_1) = \{  \{1,6\} \}\]
    \[ EF(m_2) = \{ \{2,4\}, \{3,5\} \}.\]

    Thus, the graph $G_{\langle m_1, m_2 \rangle}$ is illustrated in Figure \ref{Graph}.

  \begin{figure}[ht]
    \centering
   \begin{subfigure}{.24\textwidth}
\centering
$ \vcenter{\hbox{
\begin{overpic}[scale = 1]{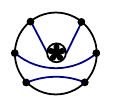}
\end{overpic} }} $
\caption{$m_1 = (25)(34)(1)(6)$.} \label{m1}
\end{subfigure}
 \begin{subfigure}{.23\textwidth}
\centering
$ \vcenter{\hbox{
\begin{overpic}[scale = 1]{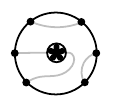}
\end{overpic} }} $
\caption{$m_2=(45)(61)(2)(3)$.} \label{m2}
\end{subfigure}
 \begin{subfigure}{.23\textwidth}
\centering
$ \vcenter{\hbox{
\begin{overpic}[scale = 1]{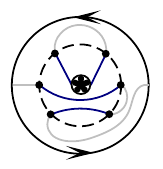}
\end{overpic} }} $
\caption{$\langle m_1, m_2\rangle = xy$.} \label{bilinearxy}
\end{subfigure}
 \begin{subfigure}{.23\textwidth}
\centering
$ \vcenter{\hbox{
\begin{overpic}[scale = 1]{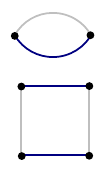}
\end{overpic} }}$
\caption{$G_{\langle m_1, m_2 \rangle}$.} \label{Graph}
\end{subfigure}
    \caption{Example of graph $G_{\langle m_1, m_2 \rangle}$.}
    \label{fig:exampleofgraph}
\end{figure}
\end{example}

\begin{definition}
  Define  $C_{m_1, m_2} = \{C_1, C_2, \dots C_k\}$ such that $\bigsqcup_{i=1}^k C_i = G$; it should be clear that $\deg(\langle m_1, m_2 \rangle) = k$. We denote by $\langle C_i \rangle$ the homotopy type of the closed curve in $Mb$ corresponding to the component $C_i$.
\end{definition}

\subsection{Identifying Curves \& Components}

Since we work only with elements of $Mb_{n,1}$, no simple closed curves in $Kb$ will have more than one arc passing through the same crosscap. Thus, identifying $x,y,w$ curves is easy:

\begin{lemma}
    Let $C$ be a component of $G_{\langle m_1, m_2 \rangle}$;

    then 
    \[ \langle C \rangle = \begin{cases}
        x & (C \cap F(m_1) \neq \emptyset) \wedge (C \cap F(m_2) = \emptyset) \\
        y & (C \cap F(m_1) = \emptyset) \wedge (C \cap F(m_2) \neq \emptyset) \\
        w & (C \cap F(m_1) \neq \emptyset) \wedge (C \cap F(m_2) \neq \emptyset) \\
    \end{cases}.\]
\end{lemma}

\begin{proof}
    The condition of $C \cap F(m_i)$ determines whether an arc passes through the internal crosscap, external crosscap, or both.
\end{proof}

For $C$ a component of $G_{\langle m_1, m_2 \rangle}$ containing no fixed points, $\langle C \rangle \in \{d,z\}$, which we disambiguate in the following way.

Assume $C$ is a component of $G_{\langle m_1, m_2 \rangle}$ and let $x_m$ be the minimum vertex in the first component of a transposition from $m_1$. Let $V_C =  x_m \ x_m'  \dots... \dots \ x_m  = \{z_i \}_{i=1}^{|C| +1}$ be the vertex walk around $C$ beginning at $x_m$, continuing in the direction of $x_m'$, and ending again at $x_m$. 

Define \[ \psi_k(i, j) = \begin{cases}
    [j-i]_{2n} & \text{if }(i,j) \in m_k \\
    [i-j]_{2n} & \text{if }(j,i) \in m_k \\
\end{cases}\]
for $k=1,2$.

\begin{proposition}[$d$ and $z$ curves]
Let $C$ be a component of $G_{\langle m_1, m_2 \rangle}$ containing no fixed points. Let $V_C = \{z_i \}_{i=1}^{|C_{m_1, m_2}| +1}$ as previously defined.

Let \[ \Psi(C) = \sum_{k=1}^{|C|}\psi_1(z_{2 k-1}, z_{2k}) + \psi_2(z_{2 k}, z_{2k+1}). \]

\begin{enumerate}
    \item If $\Psi(C) = 0$, then $\langle C \rangle = d$
    \item If $\Psi(C) = \pm 2n$, then $\langle C \rangle = z$.
\end{enumerate}

\begin{proof}
    Each transposition $(i \ j)$ from an element $m \in Mb_n$ defines an arc tracing out an angle of $\pm \frac{2 \pi}{[j-i]_n \cdot n}$ around the boundary of the M\"obius band. The vertex walk $V_C$ is oriented in accordance with the first counterclockwise angle traced by $(x_m \ x_m') \in m_1$. The function $\psi_k$ determines the direction and length of the angles swept out by every transposition.
\end{proof}

\end{proposition}

\section{Future directions}\label{FutDirec}
In \cite{IM2} it was observed that the conjectured formula for $\det (\tilde{G}_n^{Mb_{n,1}})$ is a conjectured factor of the $n$-th Gram determinant of type $Mb$, $D_n^{Mb}$. Recall that Conjecture \ref{newconjecture} was restated in Conjecture \ref{conjecturer} after applying new formulae for Chebyshev polynomials. Therefore, Qi Chen's conjecture can also be re-written in a more elegant way. Since these new restatements consist of products of Chebyshev polynomials of the second kind, it suggests that Jones-Wenzl idempotents may be used to prove Conjecture \ref{conjecturer} and Conjecture \ref{Qi}.

\begin{conjecture}[Chen] \label{Qi}\

Let $R = \mathbb{Z}[A^{\pm 1},w,x,y,z].$ Then
the Gram determinant of type {\it Mb} for $n \geq 1$, denoted by $D_n^{\mathit{Mb}}$, is:
 \begin{eqnarray*}
D^{\mathit{Mb}}_n(d,w,x,y,z) &=& \prod_{k=1}^n (T_k(d)+(-1)^kz)^{\binom{2n}{n-k} } \\
& & \prod\limits_{ \substack{k=1 \\ k\text{ odd }}}^n
%_{k=1 \atop k \text{ odd }}^n 
\left((T_k(d) - (-1)^k z)T_k(w) -2xy\right)^{\binom{2n}{n-k}} \\
& & \prod\limits_{ \substack{k=1 \\ k\text{ even }}}^n
%_{ k=1 \atop k \text{ even }}^n 
\left((T_k(d) - (-1)^kz)T_k(w)-2(2-z)\right)^{\binom{2n}{n-k}} \\
& & \prod_{i=1}^{n} D_{n,i},
\end{eqnarray*}

where $D_{n,i} = \prod\limits_{k=1+i}^n ((d^2-4)(S_{k-1}(d))^{2\binom{2n}{n-k}}$, and $i$ represents the number of curves passing through the crosscap.
\end{conjecture}

\section*{Acknowledgements}
DI was supported by the Australian Research Council, Grant DP240102350. GMV acknowledges the support of the National Science Foundation through Grant DMS-2212736 and thanks Monash University, School of Mathematics, for their invitation to visit, where the final part of this work was produced.

\end{document}